\documentclass[12pt]{amsart}
\usepackage[utf8]{inputenc}
\usepackage{hyperref}
\usepackage[dvipsnames]{xcolor}
\usepackage{enumerate}
\usepackage{graphicx}
\usepackage{float}
\usepackage{multirow}
\usepackage{lipsum}
\usepackage{mathdots}
\usepackage{fancyhdr}
\usepackage{tikz}
\usepackage{tikz-cd}
\usepackage{setspace}
\usepackage{mathtools}
\usetikzlibrary{decorations.markings}

\newtheorem{theorem}{Theorem}
\newtheorem{axiom}{Axiom}
\newtheorem{lemma}{Lemma}

\newtheorem{definition}{Definition}
\newtheorem{remark}{Remark}

\newtheorem{conjecture}{Conjecture}

\title[Retro-synthetic Riemannian geometry I]{Retro-synthetic Riemannian geometry in Rocq I : Taylor series of a metric}
\author{Gabriella Clemente}
\email{gabriella.clemente@cnrs.fr}
\urladdr{https://sites.google.com/view/gclemente}
\address{Université Paris Cité, CNRS, IRIF, Paris, France}
\author{Hugo Herbelin}
\email{hugo.herbelin@inria.fr}
\urladdr{https://pauillac.inria.fr/~herbelin/}
\address{Université Paris Cité, INRIA, IRIF, Paris, France}
\author{Carlos Simpson}
\email{carlos.simpson@univ-cotedazur.fr}
\urladdr{https://math.univ-cotedazur.fr/u/carlos/}
\address{Université C{\^o}te d'azur, CNRS, LJAD, France}
\date{}

\begin{document}

\keywords{ Synthetic mathematics, Riemannian geometry, formalization, Rocq}

\begin{abstract}
We introduce a synthetic framework for doing local Riemannian geometry that is influenced by a reverse-engineering method for formalizing in Rocq. These new perspectives and techniques lead to a synthetic Taylor formula for a Riemannian metric in normal coordinates, which is the main contribution of this work.
\end{abstract}

\maketitle

\section*{Introduction}\label{intro}

In this article, we introduce a form of synthetic Riemannian geometry (SRG) as a subject of its own and experiment with a reverse-engineering technique for formalizing it. SRG is a classical $2^{nd}$ order theory that uses the arithmetic of the reals $\mathbb{R}$. Our synthetic results can be viewed as formal counterparts of local results of Riemannian geometry (RG). The formalization is done in Rocq \cite{Rocq} and is attached to the article. It is also available on GitHub \url{https://github.com/GabriellaClemente/differential-geometry}. A salient feature of our project is how mathematically natural the code is. That being said, our goal is not to formalize all of our work in detail, but rather to exhibit a method via which this could be done optimally with the help of automation plugins and/or AI-powered ITPs, such as LLM4Rocq \cite{LLM}. We will investigate the automation aspect of our formalization in the forthcoming part II of this article, which will be on a synthetic version of Weitzenb{\"o}ck's formula for differential $k$-forms. But the present work demonstrates that despite the lack of a complete RG library, it is still possible to formalize local RG in Rocq via the synthetic approximation we develop here. 

Our treatment of SRG was influenced by the dynamic process of formalization. In fact, the term ``retro-synthetic'' refers to the process of generating a synthetic theory based on a formalization carried out in stages of increasing complexity (i.e.\ a reverse-engineered formalization). The first stages conglomerate all of the content in axioms, the mid-iterations gradually unpack axioms into statements with proofs, and the final iteration is the maximally unpacked version, hence the closest to the synthetic pen-and-paper proof. 

The classical RG theorem of interest is stated below (see \cite{Lee}, and also \cite{normalC} for a higher order generalization). Here the euclidean norm on $\mathbb{R}^n$ will be denoted by $\| \cdot \|.$ 

\begin{theorem}{(Second order Taylor series of a Riemannian metric in normal coordinates)}\label{RTHM}
Let $(X,g)$ be a Riemannian manifold, $p \in X,$ and $(x)$ be normal coordinates centered at $p.$ Let $\delta_i^j$ be the Kronecker delta function, and $R_{iklj}$ be the Riemann curvature tensor. Then,

\[
g_{ij}(x)=\delta_i^j -\sum_{1\leq k,l \leq n} \frac{1}{3} R_{iklj} x_k x_l +O(\|x\|^3).
\]
\end{theorem}

Our synthesis of this theorem is founded on a minimalistic definition of Riemannian manifold that is nonetheless powerful enough to articulate most local results of RG and render all proofs self-contained. The key observation is that in RG, local results are independent of topology. Our approach too is topology-free. Riemannian manifolds are replaced by pairs $(M,g)$, where $M$ is a set and $g$ is an $n\times n$ invertible matrix of functions $g_{ij} \colon M \to \mathbb{R}$, for $1\leq i,j \leq n$. We call any such pair $(M,g)$ a synthetic local model of Riemannian geometry or synthetic local model (SLM) when the context is clear. For any $n\in \mathbb{N}$, we denote the set of all such SLMs by $RM(n).$ 

The synthetic version of Theorem \ref{RTHM} that we prove and (partially) formalize is stated below. 

\begin{theorem}{(Synthetic Taylor series of $g$)}\label{RTHM_SM22}
For any $(M,g) \in RM(n)$ and $p_0 \in M,$ in normal coordinates $(U_{p_0},x)$, we have that 

\[g_{ij} \circ x^{-1}=\delta_i^j -\sum_{1\leq k,l \leq n} \frac{1}{3} R_{iklj} (p_0) x_k x_l +O(\|x\|^3).\] 
\end{theorem}

In addition to SLMs, we introduce the notion of synthetic model (SM) of a theorem, which essentially condenses all of the analysis used in a proof of the theorem, c.f.\ Definition~\ref{SM_def}. In our setting, the axioms absorb the existence of normal coordinates on a Riemannian manifold, which is a fact of global analysis, and reconstitute it as the defining feature of a point. For us, a point of $(M,g) \in RM(n)$ is an element of the set $M$ that carries the synthetic analogue of normal coordinates. So clearly, we are only able to tackle formal aspects of RG.

We believe that our retro-synthesis method could also be useful for generalizing classical results, and formalizing them. In a future work, we would like to apply our techniques to obtain generalizations of Theorem \ref{RTHM} to other $G$-structures, which will lead to new results in differential geometry (c.f.\ Conjecture \ref{Gconjecture}). 

The organization is as follows. Section \ref{1} introduces SMs and the theory of SLMs. Section \ref{1.1} lays the foundations for the proof of the main theorem, which is Theorem \ref{RTHM_SM22}, and SRG as a whole. Section \ref{2} introduces the category of rings of differential operators, which provides an alternative framework for the foundations of SRG. Section \ref{2.1} contains the proof of Theorem \ref{RTHM_SM22}. Section \ref{3} is about the retro-engineering technique we used to synthesize and formalize. We also discuss there some related theoretical considerations. Section \ref{4} is on future research directions.\\ 

\noindent
\textbf{Acknowledgements.} 
This work was supported by the European Research Council through the Horizon ERC Synergy grant 101167526 (MALINCA)

\section{Synthetic theory}\label{1}
Synthetic mathematics is much more organic than what it is sometimes made out to be. In fact, it is standard mathematical practice to reason synthetically in the first stages of internalization of any theory. We produce minimalist preliminary proofs of theorems, where the proofs of auxiliary results are omitted, and slowly fill in details until the main proof is complete. Synthetic mathematics does exactly that modulo the freedom to abstract assumptions to a setting that often diverges from classical mathematics, e.g.\ synthetic differential geometry. A standard reference for the subject is \cite{Kock}. See \cite{RBGC} for recent progress in the area and its formalization in Lean.

Here we define the concept of synthetic model of a theorem, and illustrate the use of such models for formalization.

\begin{definition}\label{SM_def}
A synthetic model (SM) of a theorem $\mathbf{A}$ is a formally analogous theorem $\mathbf{A}'$ and a finite list of axioms comprising the non-formal statements that go into proving $\mathbf{A}$, from which $\mathbf{A}'$ can be proven by a finite series of reasoning steps. 
\end{definition}

In our case, the reasoning steps will all be equational rewriting. Moreover, the statement of $\mathbf{A}'$ (Theorem \ref{RTHM_SM22}) is very close to that of $\mathbf{A}$ (Theorem \ref{RTHM}). However, Definition \ref{SM_def} is broad enough to include more complex cases that we will study in the future, where $\mathbf{A}$ and $\mathbf{A}'$ could be less alike, and where other proof rules may be needed.

The formal analogy alluded to in the above definition is achieved by establishing a synthetic framework. We take as a rough guide the description in \cite{synthetic}. In the context of RG, the non-formal statements correspond to the analysis and topology at play. 

Given an SM, we call the number of axioms in the list its size, and the (total) minimal number of reasoning steps its depth. We are using the word {\it total} in a literal sense; e.g.\ in our case, we must count all rewriting steps, including those nested in the proofs of lemmas. Although depth is a well-defined quantity as the minimum exists, it is not always trivial to compute it. 

Let us define the objects we will work with, which play the role of Riemannian manifolds, and specify the rules they obey.

Let $M$ be a set. Let $g_{ij} \colon M \to \mathbb{R}$ and $g^{ij} \colon M \to \mathbb{R}$ be functions indexed by pairs of natural numbers $i,j$ belonging to a set $S \subset \mathbb{N}$ such that for all $i,j \in S$, $g_{ij} = g_{ji}$, $g^{ij}=g^{ji}$ and \[\sum_{k \in S} g^{ik} g_{kj} =\delta^i_j.\] Our shorthand notation for these functions is $g \coloneqq (g_{ij})_{i,j \in S},$ respectively $g^{-1} \coloneqq (g^{ij})_{i,j \in S},$ i.e.\ we arrange functions into square symmetric matrices of (finite or infinite) dimension $|S|$. So $gg^{-1} = g^{-1} g =I$, where $I$ is the $|S| \times |S|$ identity matrix. We are mostly interested in the case when $S=\{1,\dots,n\}$. Here, the set of all functions $M \to \mathbb{R}$ will be denoted by $\mathbb{R}^M$. 

\begin{definition}
Let $n\in \mathbb{N}$. A pair $(M,g)$ consisting of a set $M$ and a matrix function $g \in \operatorname{GL}_n (\mathbb{R}^M)$, together with the structures $\Gamma$, $\nabla$ and $R$ introduced below, is called a synthetic local model (SLM) of a Riemannian manifold of dimension $n$. The collection of all such pairs $(M,g)$ will be denoted by $RM(n).$
\end{definition}

We suppose that to any $(M,g) \in RM(n)$, we can uniquely associate the tensors 

\begin{itemize}
    
    \item $\Gamma$ with components $\Gamma^k_{ij} \colon M \to \mathbb{R}$ for all $1\leq i,j, k\leq n$, and 
    
    \item $R$ with components $R_{ijkl} \colon M \to \mathbb{R}$ for all $1\leq i,j,k,l \leq n$.
\end{itemize}

In addition, we have the {\it synthetic Levi-Civita connection} $\nabla$ of $g$, which is a collection of operators $\nabla_i$, where $1\leq i \leq n$, defined from the {\it synthetic Christoffel symbols} $\Gamma^k_{ij}$ by the formulas of Axioms 
\ref{nablax} and \ref{CLeirule} below. The tensor $R$ is called the {\it synthetic Riemann curvature tensor} of $g$ and it is given by $\nabla$ via Axiom \ref{Rnablax}. Much like in classical RG, in our synthetic model $\nabla$, $\Gamma$ and $R$ are closely related to each other.

Since $\mathbb{R}^M$ is a commutative unital $\mathbb{R}$-algebra, it is in particular a ring. For any $(M,g) \in RM(n)$, any subset $V \subset M$ and any invertible on its image function $y=(y_1,\dots,y_n) \colon V \to \mathbb{R}^n$, consider the free associative non-commutative algebra over $\mathbb{R}^M$ generated by the set of symbols 

\begin{equation*}
    \begin{split}
Symb_y &\coloneqq \Big\{\frac{\partial}{\partial y_i}, \nabla_{\frac{\partial}{\partial y_i}}, g \Big| 1\leq i \leq n \Big\}\\
&\cup \{f \circ y^{-1} \mid f\in \mathbb{R}^M \},
\end{split}
\end{equation*}
which we denote as 
\[
       \mathcal{R}_y \coloneqq \mathbb{R}^M \langle Symb_y \rangle. 
\]

Note that $\{f \circ y^{-1} \mid f\in \mathbb{R}^M \}$ includes  notably $g_{ij} \circ y^{-1}$ and $\Gamma^k_{ij} \circ y^{-1}$.
    
Initially, we suppose that all statements are made in the above defined associative algebra ($\mathcal{R}_y$) and that all members of the algebra (i.e.\ all words) are definable. Later on, this algebra will be refined. This definability supposition is reflected in Rocq by taking the symbols in $Symb_y$ as parameters. Indeed, under our assumptions, we know that we can form the following words among many others : $g \frac{\partial}{\partial y_i} \frac{\partial}{\partial y_j}$; $\nabla_{\frac{\partial}{\partial y_i}} \frac{\partial}{\partial y_j}$; $\frac{\partial (f \circ y^{-1})}{\partial y_k}$ 
for all $f\colon M \to \mathbb{R}$, so in particular $\frac{\partial (g_{ij} \circ y^{-1})}{\partial y_k}$, and also $\frac{\partial}{\partial y_l} \frac{\partial (g_{ij} \circ y^{-1})}{\partial y_k}$; and $\nabla_{\frac{\partial}{\partial y_i}} (f \circ y^{-1})$ for all $f\colon M \to \mathbb{R}$. Moreover, they are all definable.

\begin{remark}\label{der_meaning}
We will follow the notational convention \[\frac{\partial^2 (g_{ij} \circ y^{-1})}{\partial y_l \partial y_k} \coloneqq \frac{\partial}{\partial y_l} \frac{\partial (g_{ij} \circ y^{-1})}{\partial y_k}.\] Although we do not use that level of generality here, the convention extends to $m>2$. Moreover, we assume that for all $f\colon M \to \mathbb{R}$, $\frac{\partial (f \circ y^{-1})}{\partial y_k}$ has the same type as $f \circ y^{-1}$, i.e.\ it is a function $y(V) \to \mathbb{R}$.
\end{remark}

As will be discussed below, we will add relations to the algebra $\mathcal{R}_y$ representing the relations, $\mathbf{R},$ that are universally true among symbols. We get a quotient algebra 
$\overline{\mathcal{R}}_y= \mathcal{R}_y/I $ dividing by the ideal $I$ generated by the set $\mathbf{R}$,
and we will work within this algebra. Recall that we have the algebra presentation 

\[
\overline{\mathcal{R}}_y = \langle Symb_y \mid \mathbf{R} \rangle.
\]

We note that it could be interesting to look at the similar algebra defined over other fields including for example fields of finite characteristic, which would give a ``local Riemannian geometry'' in finite characteristic. 

\section{Foundations of SRG}\label{1.1}

Here we discuss the axiomatics that appear in the proof of Theorem \ref{RTHM_SM22}. We isolate the axioms and results that follow because they are important in SMs of various local RG theorems that will be the subject of future works.

\begin{remark}\label{words}
All axioms, lemmas and theorems are stated within the algebra $\mathcal{R}_y$, for suitable $y$. This means that the equations in these statements are written using the algebraic operations available in $\mathcal{R}_y$. Concretely, we use the $\mathbb{R}^M$-module structure, which gives addition and scalar multiplication, and the associative multiplication (concatenation of words) in $\mathcal{R}_y$. 

For instance, on the right hand side of the formula in Axiom~\ref{Leirule}, we are using the operations of concatenation and addition of $\mathcal{R}_y$.

In Axiom \ref{CLeirule}, what might look like multiplication in the stated equation is not ; it is concatenation of words. Namely, on the left hand side, we are concatenating word $\nabla_{\frac{\partial}{\partial y_i}}$ with word $(f\circ y^{-1})\frac{\partial}{\partial y_j}$, while on the right hand side we are concatenating $\nabla_{\frac{\partial}{\partial y_i}} (f\circ y^{-1})$ with $\frac{\partial}{\partial y_j}$, $f\circ y^{-1}$ with $\nabla_{\frac{\partial}{\partial y_i}} \frac{\partial}{\partial y_j}$, and then adding these terms in $\mathcal{R}_y.$
\end{remark}

\begin{axiom}{(Leibniz rule)}\label{Leirule}
For any $(M,g) \in RM(n)$, any functions $f_1,f_2 \colon M \to \mathbb{R}$, any subset $V \subset M$ and any invertible on its image function $y=(y_1,\dots,y_n) \colon V \to \mathbb{R}^n$, 

\begin{equation*}
    \begin{split}
        &\frac{\partial ((f_1 f_2) \circ y^{-1})}{\partial y_k} = \frac{\partial ((f_1 \circ y^{-1}) (f_2 \circ y^{-1}))}{\partial y_k}\\
        &= \frac{\partial (f_1 \circ y^{-1})}{\partial y_k} (f_2 \circ y^{-1})+ (f_1 \circ y^{-1}) \frac{\partial (f_2 \circ y^{-1})}{\partial y_k}.
    \end{split}
\end{equation*}

\end{axiom}

\begin{axiom}\label{gaxiom}
For any $(M,g) \in RM(n)$, any subset $V \subset M$, any invertible on its image function $y=(y_1,\dots,y_n) \colon V \to \mathbb{R}^n$, and all $1\leq i,j \leq n$, we have that \[g_{ij} \circ y^{-1} = g \frac{\partial}{\partial y_i} \frac{\partial}{\partial y_j}.\]
\end{axiom}

\begin{axiom}{(Synthetic smoothness of $g$)}\label{ax_2}
For any $(M,g) \in RM(n),$ $p_0 \in M,$ subset $V_{p_0} \ni p_0$ of $M$, and invertible on its image function $y=(y_1,\dots,y_n) \colon V_{p_0} \to \mathbb{R}^n$, we have that 

\begin{equation*}
\begin{split}
& g_{ij} \circ y^{-1} =g_{ij} (p_0)+\sum_{1 \leq k \leq n} \big(\frac{\partial (g_{ij} \circ y^{-1})}{\partial y_k} (y(p_0))\big) (y_k-y_k(p_0))\\
& +\sum_{1 \leq k,l \leq n} \frac{1}{2}\big(\frac{\partial^2 (g_{ij} \circ y^{-1})}{\partial y_k \partial y_l}(y(p_0))\big) (y_k-y_k(p_0))(y_l-y_l(p_0))\\
& + O(\|y-y(p_0)\|^3).
\end{split}
\end{equation*}
\end{axiom}

Some clarifications are in order. We need to use Remarks \ref{der_meaning} and \ref{words} to parse the formula of Axiom \ref{ax_2}. The formula is a sum of $4$ words in $\mathcal{R}_y$ that are also functions of the same type as $g_{ij} \circ y^{-1}$. We view the first term $g_{ij} (p_0)$ as a constant function. The generic form of the second term is that of a function $\sum_{1\leq k\leq n} C_k (f_k \circ y^{-1}),$ where for each $k$, $C_k \in \mathbb{R}$ is a constant and $f_k \colon V \to \mathbb{R}$. We use Remark \ref{der_meaning} to set $C_k = \frac{\partial (g_{ij} \circ y^{-1})}{\partial y_k} (y(p_0))$, and we independently set $f_k = y_k-y_k(p_0)$. The third term can be interpreted similarly, i.e.\ define constants $C_{kl} = \frac{1}{2} \frac{\partial^2 (g_{ij} \circ y^{-1})}{\partial y_k \partial y_l}(y(p_0))$ and functions $f_{kl} = (y_k-y_k(p_0))(y_l-y_l(p_0))$. The fourth term is a function $y(V) \to \mathbb{R}$ that is $O(\|y-y(p_0)\|^3)$, meaning that we can factor out a collection of degree $3$ monomials.

\begin{axiom}\label{nablax}
For any $(M,g) \in RM(n)$, any subset $V \subset M$, any invertible on its image function $y=(y_1,\dots,y_n) \colon V \to \mathbb{R}^n$, and all $1\leq i,j \leq n$, we have that \[\nabla_{\frac{\partial}{\partial y_i}} \frac{\partial}{\partial y_j} = \sum_{1\leq k\leq n} (\Gamma^k_{ij} \circ y^{-1}) \frac{\partial}{\partial y_k}. \]
\end{axiom}

Axiom \ref{nablax} says that $\Gamma$ completely determines $\nabla,$ while Axiom \ref{Chriscurvature} below says that $\Gamma$ is completely determined by (derivatives of) $g.$

\begin{axiom}\label{Chriscurvature}
For any $(M,g) \in RM(n)$, any subset $V \subset M$, any invertible on its image function $y=(y_1,\dots,y_n) \colon V \to \mathbb{R}^n$, and all $1\leq i,j, k\leq n$, we have that

\begin{equation*}
    \begin{split}
       \Gamma^k_{ij} \circ y^{-1} &= \sum_{1 \leq m \leq n} \frac{1}{2} (g^{km} \circ y^{-1})\Big(\frac{\partial (g_{mi} \circ y^{-1})}{\partial y_j} + \frac{\partial (g_{mj} \circ y^{-1})}{\partial y_i} \\
       &- \frac{\partial (g_{ij} \circ y^{-1})}{\partial y_m}\Big). 
    \end{split}
\end{equation*}

\end{axiom}

\begin{lemma}\label{symChris}
For all $1\leq i,j, k\leq n$, we have that $\Gamma^k_{ij} = \Gamma^k_{ji}$.
\end{lemma}

\begin{proof}
Use commutativity of addition, and the symmetry of $g.$
\end{proof}

\begin{axiom}{(Combined Leibniz rule)}\label{CLeirule}
For any $(M,g) \in RM(n)$, any function $f\colon M \to \mathbb{R}$, any subset $V \subset M$ and any invertible on its image function $y=(y_1,\dots,y_n) \colon V \to \mathbb{R}^n$, 
\[\nabla_{\frac{\partial}{\partial y_i}} \big((f\circ y^{-1})\frac{\partial}{\partial y_j}\big) = \big(\nabla_{\frac{\partial}{\partial y_i}} (f\circ y^{-1})\big) \frac{\partial}{\partial y_j} + (f\circ y^{-1})  \nabla_{\frac{\partial}{\partial y_i}} \frac{\partial}{\partial y_j}.\]
\end{axiom}

\begin{axiom}\label{intpartial}
For any $(M,g) \in RM(n)$, any function $f\colon M \to \mathbb{R}$, any subset $V \subset M$ and any invertible on its image function $y=(y_1,\dots,y_n) \colon V \to \mathbb{R}^n$, 
\[\nabla_{\frac{\partial}{\partial y_i}} (f\circ y^{-1}) = \frac{\partial (f\circ y^{-1})}{\partial y_i}.\]
\end{axiom}

\begin{lemma}\label{LRwithChrisymbol}
For any $(M,g) \in RM(n)$, any function $f\colon M \to \mathbb{R}$, any subset $V \subset M$ and any invertible on its image function $y=(y_1,\dots,y_n) \colon V \to \mathbb{R}^n$, 

\begin{equation*}
    \begin{split}
       \nabla_{\frac{\partial}{\partial y_i}} \big((f\circ y^{-1})\frac{\partial}{\partial y_j}\big) &= \frac{\partial (f\circ y^{-1})}{\partial y_i} \frac{\partial}{\partial y_j} + (f\circ y^{-1}) \sum_{1\leq k\leq n} (\Gamma^k_{ij} \circ y^{-1})\frac{\partial}{\partial y_k}. 
    \end{split}
\end{equation*}

\end{lemma}

\begin{proof}
A consequence of Axioms \ref{CLeirule}, \ref{nablax} and \ref{intpartial}. 
\end{proof}

\begin{lemma}\label{gdisr}
For any $(M,g) \in RM(n)$, any functions $f_1,\dots,f_n \colon M \to \mathbb{R}$, any subset $V \subset M$ and any invertible on its image function $y=(y_1,\dots,y_n) \colon V \to \mathbb{R}^n$, 

\[g\big(\sum_{1\leq i\leq n} f_i \frac{\partial}{\partial y_i}\big)\frac{\partial}{\partial y_j} = \sum_{1\leq i\leq n} f_i (g_{ij} \circ y^{-1}).\]
\end{lemma}

\begin{proof}
Consequence of the underlying algebraic structure we assumed to have, and Axiom \ref{gaxiom}.
\end{proof}

\begin{axiom}\label{Rnablax}
For any $(M,g) \in RM(n)$, any subset $V \subset M$, any invertible on its image function $y=(y_1,\dots,y_n) \colon V \to \mathbb{R}^n$, and all $1\leq i,j,k,l \leq n$, 
\[R_{ijkl} \circ y^{-1} = g\big(\nabla_{\frac{\partial}{\partial y_i}} \nabla_{\frac{\partial}{\partial y_j}} \frac{\partial}{\partial y_k} - \nabla_{\frac{\partial}{\partial y_j}} \nabla_{\frac{\partial}{\partial y_i}} \frac{\partial}{\partial y_k} \big)\frac{\partial}{\partial y_l}.\]
\end{axiom}

\begin{lemma}\label{Rlma}
For any $(M,g) \in RM(n)$, any subset $V \subset M$, any invertible on its image function $y=(y_1,\dots,y_n) \colon V \to \mathbb{R}^n$, and all $1\leq i,j,k,l \leq n$, we have that

\begin{equation*}
    \begin{split}
        R_{ijkl} \circ y^{-1} &= \sum_{1\leq v\leq n} g_{vl} \Big[\frac{\partial (\Gamma^v_{jk} \circ y^{-1})}{\partial y_i} - \frac{\partial (\Gamma^v_{ik} \circ y^{-1})}{\partial y_j} \\
        & + \sum_{1\leq m\leq n} (\Gamma^m_{jk} \Gamma^v_{im} 
        - \Gamma^m_{ik} \Gamma^v_{jm}) \circ y^{-1}\Big].
    \end{split}
\end{equation*}

\end{lemma} 

\begin{proof}
First, observe that 

\begin{equation*}
    \begin{split}
   \nabla_{\frac{\partial}{\partial y_i}} \nabla_{\frac{\partial}{\partial y_j}} \frac{\partial}{\partial y_k} - \nabla_{\frac{\partial}{\partial y_j}} \nabla_{\frac{\partial}{\partial y_i}} \frac{\partial}{\partial y_k} &=_{ax.\ref{nablax}} \nabla_{\frac{\partial}{\partial y_i}} \big(\sum_{1 \leq m\leq n} (\Gamma^m_{jk} \circ y^{-1}) \frac{\partial}{\partial y_m} \big) 
   - \\
   &\nabla_{\frac{\partial}{\partial y_j}} \big(\sum_{1 \leq m\leq n} (\Gamma^m_{ik} \circ y^{-1}) \frac{\partial}{\partial y_m} \big)\\
   &= \sum_{1 \leq m\leq n} \Big[\nabla_{\frac{\partial}{\partial y_i}} \big((\Gamma^m_{jk}\circ y^{-1}) \frac{\partial}{\partial y_m}\big)-\\ &\nabla_{\frac{\partial}{\partial y_j}} \big((\Gamma^m_{ik}\circ y^{-1}) \frac{\partial}{\partial y_m}\big)\Big]\\
   &=_{lem.\ref{LRwithChrisymbol}} \sum_{1 \leq m\leq n} \Big[\frac{\partial (\Gamma^m_{jk} \circ y^{-1})}{\partial y_i} \frac{\partial}{\partial y_m} +\\ 
   &\sum_{1\leq v\leq n} (\Gamma^m_{jk} \Gamma^v_{im} \circ y^{-1}) \frac{\partial}{\partial y_v}- \frac{\partial (\Gamma^m_{ik} \circ y^{-1})}{\partial y_j} \frac{\partial}{\partial y_m}-\\ &\sum_{1\leq v\leq n} (\Gamma^m_{ik} \Gamma^v_{jm}\circ y^{-1}) \frac{\partial}{\partial y_v} \Big]\\
   &= \sum_{1 \leq v\leq n} \Big[\frac{\partial (\Gamma^v_{jk} \circ y^{-1})}{\partial y_i} 
   - \frac{\partial (\Gamma^v_{ik} \circ y^{-1})}{\partial y_j}+ \\
   & \sum_{1\leq m\leq n} (\Gamma^m_{jk} \Gamma^v_{im} - \Gamma^m_{ik} \Gamma^v_{jm})\circ y^{-1}\Big] \frac{\partial}{\partial y_v}.
    \end{split}
\end{equation*}

Put \[f_v = \frac{\partial (\Gamma^v_{jk} \circ y^{-1})}{\partial y_i} - \frac{\partial (\Gamma^v_{ik} \circ y^{-1})}{\partial y_j}+ \sum_{1\leq m\leq n} (\Gamma^m_{jk} \Gamma^v_{im} - \Gamma^m_{ik} \Gamma^v_{jm})\circ y^{-1}\] so that \[\nabla_{\frac{\partial}{\partial y_i}} \nabla_{\frac{\partial}{\partial y_j}} \frac{\partial}{\partial y_k} - \nabla_{\frac{\partial}{\partial y_j}} \nabla_{\frac{\partial}{\partial y_i}} \frac{\partial}{\partial y_k} = \sum_{1 \leq v\leq n} f_v \frac{\partial}{\partial y_v}.\] Now observe that 

\begin{equation*}
    \begin{split}
        R_{ijkl} \circ y^{-1} &=_{ax.\ref{Rnablax}} g \Big( \sum_{1 \leq v\leq n} f_v \frac{\partial}{\partial y_v} \Big) \frac{\partial}{\partial y_l} \\
   &=_{lem.\ref{gdisr}} \sum_{1 \leq v\leq n} f_v (g_{vl} \circ y^{-1}) \\
   &= \sum_{1\leq v\leq n} (g_{vl} \circ y^{-1}) \Big[\frac{\partial (\Gamma^v_{jk} \circ y^{-1})}{\partial y_i} - \frac{\partial (\Gamma^v_{ik} \circ y^{-1})}{\partial y_j} \\
        & + \sum_{1\leq m\leq n} (\Gamma^m_{jk} \Gamma^v_{im} 
        - \Gamma^m_{ik} \Gamma^v_{jm}) \circ y^{-1}\Big].
    \end{split}
\end{equation*}
\end{proof}

\begin{axiom}{(Synthetic Bianchi identity)}\label{BI}
For all $1\leq i,j,k,l \leq n,$ \[R_{ijkl} + R_{jkil} + R_{kijl} = 0.\]
\end{axiom}

\begin{axiom}\label{last2in}
For all $1\leq i,j,k,l \leq n,$ $R_{ijkl} = -R_{ijlk}.$
\end{axiom}

\begin{lemma}\label{R_sym}
For all $1\leq i,j,k,l \leq n$, we have that 
\begin{enumerate}
    \item $R_{ijkl} = -R_{jikl},$ and
    \item $R_{ijkl} = R_{klij}$.
\end{enumerate}
\end{lemma}

\begin{proof}
For $(1)$, simply use Axiom \ref{Rnablax}. For $(2)$, we use the Bianchi identity, $(1)$ and Axiom \ref{last2in}.
The calculations are similar to the above. 
\end{proof}

Let us observe that our initial algebraic context has been narrowed down to the infinitely generated subalgebra of $\mathcal{R}_y$ with presentation \[\langle Symb_y \mid \mbox{ax., lem.} \rangle,\] where ax.\ and lem.\ stand for the equations appearing in all axioms and lemmas of this subsection.

\section{Rings of differential operators for SRG}\label{2}

We discuss an alternative, and perhaps more elegant point of view for the foundations of SRG, which is based on the theory of $D$-modules. We also outline how to recast our foundational development from sections \ref{1} -- \ref{1.1}. Some advantages include the reduction in size of the SM of Theorem \ref{RTHM}, and the freedom to do local differential geometry over any base ring (e.g.\ a field of finite characteristic). 

Fix a base commutative ring $k$. Define a {\it ring of differential operators} (over $k$) to be an associative $k$-algebra $R$ together with a filtration $R=\bigcup _{k\geq 0} R_k$ with $k\subset R_0$, satisfying $R_iR_j \subset R_{i+j}$, such that 
\[
Gr(R) \coloneqq \bigoplus _{i\geq 0} R_i/R_{i-1}
\]
is commutative (here we set $R_{-1}:= \{ 0\}$). Note that $R_0$ is a commutative subalgebra of $R$. A {\it morphism} $\phi  : R\rightarrow R'$ is a morphism of $k$-algebras such that $\phi (R_i)\subset R'_i$. For any $j \geq -1$, we call $R_j$ the subspace of differential operators {\it of order $\leq j$}. 

The category ${\bf RDO}$ of rings of differential operators admits all colimits so we may define rings of differential operators by generators and relations. For our purposes here we will be interested in the ones that are {\it generated in orders $\leq 1$}, that is to say $R$ is generated by $R_1$ and the filtration is generated in the sense that $R_m$ is the subspace of products of length $\leq m$ of elements of $R_1$. We also have a Leibniz rule valid in any $R\in {\bf RDO}$, which applies specifically to $R_1$ acting on $R_0$ by derivations.

If $A$ is a commutative $k$-algebra then $A\langle \partial _1,\ldots , \partial _k\rangle$
is defined to be the object of ${\bf RDO}$ generated by $A$ in degree $0$ and the $\partial _i$ in degree $1$, subject to the relation $[\partial _i,\partial _j]=\partial _i \partial _j - \partial _j \partial _i = 0$. 

For {\it local RG} in dimension $n$, let $k={\mathbb R}$, let 
${\mathcal A}$ be the commutative ${\mathbb R}$-algebra generated by $n(n+1)/2$ elements denoted $g_{ij}$ and  $n(n+1)/2$ elements denoted $g^{ij}$ (with the conventions $g_{ij}=g_{ji}$ and $g^{ij} = g^{ji}$) subject to the relations
\[
\sum _{j=1}^n g^{ij} g_{jk} = \delta^i_k.
\]
Then consider the object $RG$ of ${\bf RDO}$ defined by 
\[
RG \coloneqq {\mathcal A}\langle \partial _1,\ldots , \partial _n\rangle .
\]
Thus, concretely, $RG$ is generated by elements $g_{ij}$, $g^{ij}$ and $\partial _i$ 
subject to the relations $\sum _{j=1}^n g^{ij} g_{jk} = \delta^i_k$, 
$\partial _i \partial _j = \partial _j \partial _i$, plus the relations imposed by the conditions of ${\bf RDO}$. The main objects entering into local Riemannian geometry may be seen as elements of the ring $RG$, when expressed using index notation in terms of local coordinates. This  includes for example the Christoffel symbols given by the formula of Axiom \ref{Chriscurvature}
\[
\Gamma ^i_{jk} = \frac{1}{2}\sum _m g^{im}\left(
\partial_k(g_{mj}) + \partial_j(g_{mk}) - \partial_m(g_{jk})
\right) ,
\]
as well as the curvature (Lemma \ref{Rlma}) deduced from the Christoffel symbols. Note that in order for the present framework to be fully consistent with the statements of Axiom \ref{Chriscurvature} and Lemma \ref{Rlma}, all occurrences of $y$ have been removed. Conceptually, we can replace $\mathcal{R}_y$, where $y$ is an arbitrary function of the necessary kind, by $RG$ to obtain a practically equivalent foundation of SRG.

\section{Synthetic model of Taylor's theorem for Riemannian metrics}\label{2.1}

This section is dedicated to the proof of Theorem \ref{RTHM_SM22}. As mentioned in the introduction, a point on $(M,g)$ is defined by normal coordinates, which are given as in Axiom \ref{ax_1'} and Axiom \ref{Christosum} below. More precisely, a point is defined by the specific properties of classical normal coordinates that are used in the proof of Theorem \ref{RTHM_SM22}. But we do not claim that those properties alone are enough to define normal coordinates. Also, we should point out that the properties used may change based on the result to be proven. In this way, SMs of theorems are dynamic. Note that in particular, for us, points carry metric data.

\begin{axiom}{(Synthetic normal coordinates -- I)}\label{ax_1'}
For any $(M,g) \in RM(n)$ and $p_0 \in M,$ there is a subset of $M$, $U_{p_0} \ni p_0$, and an invertible on its image function $x=(x_1,\dots,x_n) \colon U_{p_0} \to \mathbb{R}^n$ such that

\begin{enumerate}
    \item $x(p_0)=0,$
    \item for all $1 \leq i,j \leq n$, $g_{ij} (p_0) =\delta^j_i,$ 
    \item for all $1 \leq i,j,k \leq n$, $\frac{\partial (g_{ij} \circ x^{-1})}{\partial x_k} (x(p_0))=0,$ and
    \item for all $1 \leq i,j,k \leq n$, $(\Gamma^k_{ij} \circ x^{-1})(x(p_0)) = \Gamma^k_{ij} (p_0) = 0.$
\end{enumerate}
\end{axiom}

A pair $(U_{p_0},x)$ of subset $U_{p_0}$ and function $x$ as described in Axiom \ref{ax_1'} is called a normal coordinate chart. Note that \textsf{$\forall p \in U_{p_0}$} in Rocq codifies the same data as $x$ in the article. We discuss normal coordinates in the ${\bf RDO}$ sense at the end of this section. 

\begin{lemma}\label{Christocurv}
For any $(M,g) \in RM(n)$ and $p_0 \in M,$ we have, in normal coordinates $(U_{p_0},x)$, that \[R_{klij} (p_0) = \frac{\partial (\Gamma_{jk}^l \circ x^{-1})}{\partial x_i}(x(p_0)) - \frac{\partial (\Gamma_{ik}^l \circ x^{-1})}{\partial x_j}(x(p_0)).\]
\end{lemma}

\begin{proof}
Apply $2$ and $4$ of Axiom \ref{ax_1'} to Lemma \ref{Rlma}.
\end{proof}

\begin{axiom}{(Synthetic normal coordinates -- II)}\label{Christosum}
For any $(M,g) \in RM(n)$ and $p_0 \in M,$ we have, in normal coordinates $(U_{p_0},x)$, that 
\[
\frac{\partial (\Gamma_{ij}^k \circ x^{-1})}{\partial x_l}(x(p_0))+\frac{\partial (\Gamma_{il}^k \circ x^{-1})}{\partial x_j}(x(p_0))+\frac{\partial (\Gamma_{jl}^k \circ x^{-1})}{\partial x_i}(x(p_0))=0.
\]
\end{axiom}

\begin{lemma}\label{Christocurv_lem}
For any $(M,g) \in RM(n)$ and $p_0 \in M,$ we have, in normal coordinates $(U_{p_0},x)$, that \[R_{klij} (p_0) = - \Big(\frac{\partial (\Gamma_{ij}^l \circ x^{-1})}{\partial x_k} +2 \frac{\partial (\Gamma_{ik}^l \circ x^{-1})}{\partial x_j}\Big)(x(p_0)).\]
\end{lemma}

\begin{proof}
According to Lemma \ref{Christocurv}, we have that \[R_{klij} (p_0) = \frac{\partial (\Gamma_{jk}^l \circ x^{-1})}{\partial x_i}(x(p_0)) - \frac{\partial (\Gamma_{ik}^l \circ x^{-1})}{\partial x_j}(x(p_0)).\] According to Axiom \ref{Christosum}, upon exchanging $k \leftrightarrow l$, we have that 

\begin{equation}\label{eqC}
    -\frac{\partial (\Gamma_{jk}^l \circ x^{-1})}{\partial x_i}(x(p_0))=\frac{\partial (\Gamma_{ij}^l \circ x^{-1})}{\partial x_k}(x(p_0))+\frac{\partial (\Gamma_{ik}^l \circ x^{-1})}{\partial x_j}(x(p_0)).
\end{equation}

Hence,

\begin{equation*}
    \begin{split}
        R_{klij} (p_0) &=_{lem. \ref{Christocurv}} \frac{\partial (\Gamma_{jk}^l \circ x^{-1})}{\partial x_i}(x(p_0)) - \frac{\partial (\Gamma_{ik}^l \circ x^{-1})}{\partial x_j}(x(p_0)) \\
        &=-\Big(\frac{\partial (\Gamma_{ik}^l \circ x^{-1})}{\partial x_j}(x(p_0))-\frac{\partial (\Gamma_{jk}^l \circ x^{-1})}{\partial x_i}(x(p_0))\Big)\\
        &=_{(\ref{eqC})} -\Big(\frac{\partial (\Gamma_{ik}^l \circ x^{-1})}{\partial x_j}(x(p_0)) + \frac{\partial (\Gamma_{ij}^l \circ x^{-1})}{\partial x_k}(x(p_0))+\\
        & \frac{\partial (\Gamma_{ik}^l \circ x^{-1})}{\partial x_j}(x(p_0))\Big)\\
        &=-\Big(\frac{\partial (\Gamma_{ij}^l \circ x^{-1})}{\partial x_k}(x(p_0))+2\frac{\partial (\Gamma_{ik}^l \circ x^{-1})}{\partial x_j}(x(p_0))\Big).
    \end{split}
\end{equation*}
\end{proof}

\begin{axiom}\label{pargax}
For any $(M,g) \in RM(n)$, any subset $V \subset M$ and any invertible on its image function $y=(y_1,\dots,y_n) \colon V \to \mathbb{R}^n$, 
\[
\frac{\partial (g_{ij} \circ y^{-1})}{\partial y_k} = \sum_{1\leq m \leq n} \big[(g_{mj}\circ y^{-1}) (\Gamma^m_{ki} \circ y^{-1}) + (g_{im} \circ y^{-1}) (\Gamma^m_{kj} \circ y^{-1}) \big].
\]
\end{axiom}

\begin{lemma}\label{doubderiv}
For any $(M,g) \in RM(n)$ and $p_0 \in M,$ we have, in normal coordinates $(U_{p_0},x)$, that
\[\frac{\partial^2 (g_{ij}\circ x^{-1})}{\partial x_l x_k} (x(p_0))= \big(\frac{\partial (\Gamma_{ki}^j \circ x^{-1})}{\partial x_l}  +\frac{\partial (\Gamma_{kj}^i \circ x^{-1})}{\partial x_l}\big)(x(p_0)).\]
\end{lemma}

\begin{proof}
Instantiate Axiom \ref{pargax} with normal coordinates $(U_{p_0},x)$ : 
\[
\frac{\partial (g_{ij} \circ x^{-1})}{\partial x_k} = \sum_{1\leq m \leq n} \big[(g_{mj}\circ x^{-1}) (\Gamma^m_{ki} \circ x^{-1}) + (g_{im} \circ x^{-1}) (\Gamma^m_{kj} \circ x^{-1}) \big].
\]

Apply the Leibniz rule and Axiom \ref{ax_1'} :

\begin{equation*}
    \begin{split}
        \frac{\partial^2 (g_{ij} \circ x^{-1})}{\partial x_l \partial x_k} (x(p_0)) &= \sum_{1\leq m \leq n} \frac{\partial}{\partial x_l}\big[(g_{mj}\circ x^{-1}) (\Gamma^m_{ki} \circ x^{-1})  \\
        + & (g_{im} \circ x^{-1}) (\Gamma^m_{kj} \circ x^{-1}) \big] (x(p_0))\\
        =_{ax.\ref{Leirule}} &\sum_{1\leq m \leq n} \big[\frac{\partial (g_{mj} \circ x^{-1})}{\partial x_l} (\Gamma^m_{ki} \circ x^{-1}) \\
        +& (g_{mj} \circ x^{-1}) \frac{\partial (\Gamma^m_{ki} \circ x^{-1})}{\partial x_l}\big] (x(p_0))\\
        +&\sum_{1\leq m \leq n} \big[\frac{\partial (g_{im} \circ x^{-1})}{\partial x_l} (\Gamma^m_{kj} \circ x^{-1})  \\
        + & (g_{im} \circ x^{-1}) \frac{\partial (\Gamma^m_{kj} \circ x^{-1})}{\partial x_l}\big] (x(p_0))\\
        =_{ax.\ref{ax_1'}.3} & \sum_{1\leq m \leq n} \big[(g_{mj} \circ x^{-1}) \frac{\partial (\Gamma^m_{ki} \circ x^{-1})}{\partial x_l} \\
        + & (g_{im} \circ x^{-1}) \frac{\partial (\Gamma^m_{kj} \circ x^{-1})}{\partial x_l}\big] (x(p_0))\\
        =_{ax.\ref{ax_1'}.2} & \big(\frac{\partial (\Gamma_{ki}^j \circ x^{-1})}{\partial x_l}  +\frac{\partial (\Gamma_{kj}^i \circ x^{-1})}{\partial x_l}\big)(x(p_0)).
    \end{split}
\end{equation*}
\end{proof}

\begin{lemma}\label{lemfinal}
For any $(M,g) \in RM(n)$ and $p_0 \in M,$ we have, in normal coordinates $(U_{p_0},x)$, that
    \[\sum_{1\leq k,l \leq n} 3\frac{\partial^2 (g_{ij} \circ x^{-1})}{\partial x_k \partial x_l} (x(p_0))x_k x_l = \sum_{1\leq k,l \leq n} 2 R_{ikjl} (p_0) x_k x_l.\]
\end{lemma}

\begin{proof}
By Axiom \ref{last2in} and Lemma \ref{R_sym}, for any given $k,l$, we can write \[2R_{ikjl}(p_0) = -(R_{kijl}(p_0)+R_{ljik}(p_0)).\] By Lemma \ref{Christocurv_lem}, we have that \[R_{kijl}(p_0) = -\Big(\frac{\partial (\Gamma_{jl}^i \circ x^{-1})}{\partial x_k} +2 \frac{\partial (\Gamma_{jk}^i \circ x^{-1})}{\partial x_l}\Big)(x(p_0))\] and \[R_{ljik}(p_0) = -\Big(\frac{\partial (\Gamma_{ik}^j \circ x^{-1})}{\partial x_l} +2 \frac{\partial (\Gamma_{il}^j \circ x^{-1})}{\partial x_k}\Big)(x(p_0)).\] Hence,

\begin{equation}\label{2R}
\begin{split}
2R_{ikjl}(p_0) &=- \Big[-\Big(\frac{\partial (\Gamma_{jl}^i \circ x^{-1})}{\partial x_k} +2 \frac{\partial (\Gamma_{jk}^i \circ x^{-1})}{\partial x_l}\Big)(x(p_0)) \\
&- \Big(\frac{\partial (\Gamma_{ik}^j \circ x^{-1})}{\partial x_l} +2 \frac{\partial (\Gamma_{il}^j \circ x^{-1})}{\partial x_k}\Big)(x(p_0))\Big]\\
&=\Big(\frac{\partial (\Gamma_{jl}^i \circ x^{-1})}{\partial x_k} +2 \frac{\partial (\Gamma_{jk}^i \circ x^{-1})}{\partial x_l}\Big)(x(p_0))+ \\
&\Big(\frac{\partial (\Gamma_{ik}^j \circ x^{-1})}{\partial x_l} +2 \frac{\partial (\Gamma_{il}^j \circ x^{-1})}{\partial x_k}\Big)(x(p_0)).
\end{split}
\end{equation}

Since we can rename the indices in a summation ($k\leftrightarrow l$), it follows that

\begin{equation}\label{xkxl}
    \begin{split}
        & \sum_{1\leq k,l \leq n} \Big(\frac{\partial (\Gamma_{ik}^j \circ x^{-1})}{\partial x_l} +2 \frac{\partial (\Gamma_{il}^j \circ x^{-1})}{\partial x_k}\Big)(x(p_0)) x_l x_k \\
        &= \sum_{1\leq k,l \leq n} \Big(\frac{\partial (\Gamma_{il}^j \circ x^{-1})}{\partial x_k} +2 \frac{\partial (\Gamma_{ik}^j \circ x^{-1})}{\partial x_l}\Big)(x(p_0)) x_k x_l.
    \end{split}
\end{equation}

Thus, 

\begin{equation*}
\begin{split}
\sum_{1\leq k,l \leq n} 2R_{ikjl}(p_0) x_k x_l &=_{(\ref{2R})}
\sum_{1\leq k,l \leq n}\Big(\frac{\partial (\Gamma_{jl}^i \circ x^{-1})}{\partial x_k} + 
2 \frac{\partial (\Gamma_{jk}^i \circ x^{-1})}{\partial x_l}\Big)(x(p_0)) x_k x_l  \\
&+\sum_{1\leq k,l \leq n} \Big(\frac{\partial (\Gamma_{ik}^j \circ x^{-1})}{\partial x_l} + 2 \frac{\partial (\Gamma_{il}^j \circ x^{-1})}{\partial x_k}\Big)(x(p_0)) x_k x_l \\
&=_{(\ref{xkxl})} \sum_{1\leq k,l \leq n}\Big(\frac{\partial (\Gamma_{jl}^i \circ x^{-1})}{\partial x_k} + 2 \frac{\partial (\Gamma_{jk}^i \circ x^{-1})}{\partial x_l}\Big)(x(p_0)) x_k x_l \\
&+ \sum_{1\leq k,l \leq n} \Big(\frac{\partial (\Gamma_{il}^j \circ x^{-1})}{\partial x_k} + 2 \frac{\partial (\Gamma_{ik}^j \circ x^{-1})}{\partial x_l}\Big)(x(p_0)) x_k x_l \\
&= \sum_{1\leq k,l \leq n}  \Big(\frac{\partial (\Gamma_{jl}^i \circ x^{-1})}{\partial x_k} + \frac{\partial (\Gamma_{il}^j \circ x^{-1})}{\partial x_k}\Big)x_k x_l \\
&+ \sum_{1\leq k,l \leq n} 2\Big(\frac{\partial (\Gamma_{jk}^i \circ x^{-1})}{\partial x_l} (x(p_0)) +\frac{\partial (\Gamma_{ik}^j \circ x^{-1})}{\partial x_l} (x(p_0))\Big)x_k x_l \\
& =_{lem.\ref{symChris}} \sum_{1\leq k,l \leq n}  \Big(\frac{\partial (\Gamma_{lj}^i \circ x^{-1})}{\partial x_k} + \frac{\partial (\Gamma_{li}^j \circ x^{-1})}{\partial x_k}\Big)x_k x_l \\
&+\sum_{1\leq k,l \leq n} 2\Big(\frac{\partial (\Gamma_{kj}^i \circ x^{-1})}{\partial x_l} (x(p_0))+\frac{\partial (\Gamma_{ki}^j \circ x^{-1})}{\partial x_l} (x(p_0))\Big)x_k x_l \\
&=_{lem. \ref{doubderiv}} \sum_{1\leq k,l \leq n}  \Big(\frac{\partial (\Gamma_{lj}^i \circ x^{-1})}{\partial x_k} + \frac{\partial (\Gamma_{li}^j \circ x^{-1})}{\partial x_k}\Big)x_k x_l \\
&+\sum_{1\leq k,l \leq n} 2\frac{\partial^2 (g_{ij}\circ x^{-1})}{\partial x_l x_k} (x(p_0)) x_k x_l \\
&=_{k \leftrightarrow l} \sum_{1\leq k,l \leq n}  \Big(\frac{\partial (\Gamma_{kj}^i \circ x^{-1})}{\partial x_l} + \frac{\partial (\Gamma_{ki}^j \circ x^{-1})}{\partial x_l}\Big)x_l x_k \\
&+ \sum_{1\leq k,l \leq n} 2\frac{\partial^2 (g_{ij}\circ x^{-1})}{\partial x_l x_k} (x(p_0)) x_k x_l   \\
&=_{x_l x_k = x_k x_l} \sum_{1\leq k,l \leq n}  \Big(\frac{\partial (\Gamma_{kj}^i \circ x^{-1})}{\partial x_l}  
+ \frac{\partial (\Gamma_{ki}^j \circ x^{-1})}{\partial x_l}\Big)x_k x_l \\
&+\sum_{1\leq k,l \leq n} 2\frac{\partial^2 (g_{ij}\circ x^{-1})}{\partial x_l x_k} (x(p_0)) x_k x_l
\end{split}
\end{equation*}

\begin{equation*}
    \begin{split}
        &=_{lem. \ref{doubderiv}}\sum_{1\leq k,l \leq n} \frac{\partial^2 (g_{ij}\circ x^{-1})}{\partial x_l x_k} (x(p_0)) x_k x_l+ \sum_{1\leq k,l \leq n} 2\frac{\partial^2 (g_{ij}\circ x^{-1})}{\partial x_l x_k} (x(p_0)) x_k x_l \\
&=\sum_{1\leq k,l \leq n} 3\frac{\partial^2 (g_{ij}\circ x^{-1})}{\partial x_l x_k} (x(p_0)) x_k x_l. 
    \end{split}
\end{equation*}
\end{proof}

We are now able to deduce the proof of the synthetic Taylor theorem for $g.$\\

\noindent
{\it Proof of Theorem \ref{RTHM_SM22}}. By Lemma \ref{lemfinal}, \[\sum_{1\leq k,l \leq n} \frac{1}{2} \frac{\partial^2 (g_{ij} \circ x^{-1})}{\partial x_k \partial x_l} (x(p_0))x_k x_l = \sum_{1\leq k,l \leq n} \frac{1}{3} R_{ikjl} (p_0) x_k x_l.\] So Axiom \ref{ax_2} in normal coordinates $(U_{p_0},x)$ reads (c.f.\ $1 - 3$ of Axiom \ref{ax_1'}),

\begin{equation*}
\begin{split}
g_{ij} \circ x^{-1} &=\delta^j_i+\sum_{1 \leq k,l \leq n} \frac{1}{2}\big(\frac{\partial^2 (g_{ij} \circ x^{-1})}{\partial x_k \partial x_l}(x(p_0))\big)x_k x_l +O(\|x\|^3).
\end{split}
\end{equation*}

From Axiom \ref{last2in}, it follows that

\begin{equation*}
\begin{split}
g_{ij} \circ x^{-1} &=\delta^j_i - \sum_{1\leq k,l \leq n} \frac{1}{3} R_{iklj} (p_0) x_k x_l +O(\|x\|^3).
\end{split}
\end{equation*} \hfill \qedsymbol

\begin{remark}\label{thmB}
The maximal scope of our axioms, i.e.\ the proposed synthetic model of Theorem \ref{RTHM}, is any proposition asserting the equality of tensor fields involving derivatives of a Riemannian metric of second order. Using Definition \ref{coords} that we discuss next, this may be extended to derivatives of any order. 
\end{remark}

We conclude with some further remarks about normal coordinates. 

Define an element of ${\bf RDO}$ called the {\it ring of coordinates} denoted $XG$ generated by $x_1,\ldots , x_n$ and $\partial _1,\ldots , \partial _n$ subject to the conditions that $x_i$ commute among themselves, $\partial _i$ commute among themselves, and $\partial _i(x_j) = \delta^j_i$ (here as usual $\partial _i(x_j)$  means the commutator $\partial _i x_j - x_j \partial _i$). 

\begin{definition}
\label{coords}
A coordinate chart is a morphism 
\[
RG \sqcup _{D} XG \stackrel{b}{\rightarrow} {\mathcal B} \stackrel{\varepsilon}{\rightarrow} {\mathbb R}
\]
to an augmented object of {\bf RDO} with $(\varepsilon \circ b)(x_i)=0$, where $D=k[\partial_1,\dots,\partial_n]$.
We say that a coordinate chart is normal
if the images of 
\[
\sum _{j,k} \Gamma ^i_{jk}x_j x_k
\]
vanish in ${\mathcal B}$, and if the image of $g_{ij}$ by $\varepsilon \circ b$ 
is $\delta^j_i$. 
\end{definition}

One can similarly define a notion of change of coordinates and a {\it tensorial quantity} is one that changes according to the usual rules of changes of coordinates.

\begin{theorem}
\label{normal}
Suppose $A$ is a tensorial quantity, and suppose $A$ vanishes in any normal coordinate chart. Then $A$ vanishes.
\end{theorem}

Suppose $U\subset {\mathbb R}^n$. The algebra ${\bf DO}(U)$ of smooth differential operators on $U$ is an object of ${\bf RDO}$ generated  in order $\leq 1$ 
by the smooth functions on $U$ and the partial derivatives $\partial _i$. If $(g_{ij})$ is a Riemannian metric on $U$, then we obtain a {\it model}
\[
RG \rightarrow {\bf DO}(U).
\]
We call this a {\it local Riemannian model}. \\

\noindent
{\it Proof of Theorem \ref{normal}}. We can give a ``classical'' proof in the following way: we note first that if $A$ vanishes under any local Riemannian model, then it vanishes. This may be seen from the fact that elements of $RG$ can be put in a unique normal form (within $RG$) 
of the form $F\cdot \partial _{i_1} \cdots \partial _{i_k}$ where 
$F\in RG_0$, and furthermore the coefficients $F$ are polynomials in derivatives of the
$g_{ij}$ and $g^{ij}$. The derivatives are subject to identities due to the commutation of the $\partial_i$ and the inverse relation between $g_{ij}$ and $g^{ij}$. 
However, subject to these identities, any combination of values 
of finitely many higher derivatives of the
$g_{ij}$ and $g^{ij}$ may be found in a local Riemannian model. 
Thus, if $A$ vanishes in any local Riemmanian model then it vanishes. Notice also that shrinking $U$ does not affect the validity of this. 

But on the other hand, we have the classical fact in Riemannian geometry that any local Riemannian metric admits locally a system of normal coordinates. These will provide a system of normal coordinates according to our Definition \ref{coords}. If $A$ is tensorial, it means that it transforms in the right way under change of coordinates. Thus, if it vanishes in any 
system of normal coordinates as defined by Definition \ref{coords}, then it will in particular vanish in a local Riemannian model with normal coordinates, and by the coordinate change formula, it vanishes in the original coordinate chart before changing to normal coordinates. $\qed$

\section{Retro-engineering in Rocq}\label{3}

\subsection{Formalization iterations}\label{3.1}

 As mentioned in the introduction we have only partially formalized the final iteration, which is contained in section~\ref{2.1}, i.e.\ we have formalized up to an earlier iteration which is not maximally unpacked. Here we discuss the first and last formalized iterations. The corresponding proof files are attached to the article. But they can also be accessed at \url{https://github.com/GabriellaClemente/differential-geometry/tree/cpp27}.

In the first iteration, we formalized an SM of size and depth equal to $2$. Points on a SLM were defined in the following way. 

\begin{axiom}\label{ax_1}
For any $(M,g) \in RM(n)$ and $p_0 \in M,$ there is a subset of $M$, $U_{p_0} \ni p_0$, and an invertible on its image function $x=(x_1,\dots,x_n) \colon U_{p_0} \to \mathbb{R}^n$ such that

\begin{enumerate}
    \item $x(p_0)=0,$
    \item for all $1 \leq i,j \leq n$, $g_{ij} (p_0) =\delta^j_i,$
    \item for all $1 \leq i,j,k \leq n$, $\frac{\partial (g_{ij} \circ x^{-1})}{\partial x_k} (x(p_0))=0,$ and 
    \item for all $1 \leq i,j,k,l \leq n,$
    \[\frac{1}{2} \frac{\partial^2 (g_{ij} \circ x^{-1})}{\partial x_k \partial x_l} (x(p_0))x_k x_l = -\frac{1}{3} R_{iklj} (p_0) x_k x_l.\]
\end{enumerate}
\end{axiom}

The second axiom was Axiom \ref{2}. Theorem \ref{RTHM_SM22} may now be proved by direct application of Axiom \ref{ax_1} to Axiom \ref{ax_2}. 

The last iteration we formalized makes use of a larger SM. In this case, a point on an SLM is defined by Axiom \ref{ax_11} below and Axiom \ref{Christosum}. 

\begin{axiom}{(Synthetic normal coordinates -- I, last iteration)}\label{ax_11}
For any $(M,g) \in RM(n)$ and $p_0 \in M,$ there is a subset of $M$, $U_{p_0} \ni p_0$, and an invertible on its image function $x=(x_1,\dots,x_n) \colon U_{p_0} \to \mathbb{R}^n$ such that

\begin{enumerate}
    \item $x(p_0)=0,$
    \item for all $1 \leq i,j \leq n$, $g_{ij} (p_0) =\delta^j_i,$ 
    \item for all $1 \leq i,j,k \leq n$, and $\frac{\partial (g_{ij} \circ x^{-1})}{\partial x_k} (x(p_0))=0$. 
\end{enumerate}
\end{axiom}

\subsection{Theoretical considerations}\label{3.2}
Proof assistants provide a formal language in which to write mathematical definitions, theorems and proofs and ensure their mathematical correctness. A proof assistant such as Rocq relies on different components. At the core is a formal logic, a type theory called the Calculus of Inductive Constructions in the case of Rocq, in which only correct definitions, statements and proofs can be expressed. Such a logic is decidable in the sense that the correctness of definitions, statements and proofs can be testified with a terminating algorithm. On top of the internal formal logic, a high level formal language can be found, which allows to phrase definitions, theorems and proofs in a way that is close to how it is done in practice. For instance, the high level formal language supports notations and various forms of implicit information.

In their current state of the art, proof assistants are not able to directly understand mathematics the way mathematicians write it, say in LaTeX, at least when not helped with artificial intelligence tools. Our work comes nowhere near solving this linguistics-informatics problem, however SMs of theorems shine a light on its nature. We believe (up to definitions that yet need to become precise) that\\

\noindent
{\bf Slogan.} The only mathematics that can be {\it formalized} is {\it synthetic}. \\

\section{Future directions}\label{4}

Below are some ideas that could lead to interesting research projects.\\

\noindent
\textbf{1. Computational complexity.} Recall Definition \ref{SM_def} and the concepts introduced in the paragraphs below it. Given an SM of size $m$ and depth $n$, how does $n$ grow with respect to $m$, e.g. is $n = O(m^2)$, or what rate best describes the problem ?\\

\noindent
\textbf{2. Series representations of synthetic $\mathbf{G}$-structures.} Let us quickly review the classical theory of $G$-structures. Most structures on smooth manifolds are $G$-strucures, e.g.\ a Riemannian metric on an $n$-dimensional manifold is such a structure for $G=O(n)$. A $G$-structure determines a compatible with it connection on the tangent bundle. For example, in the $O(n)$ case, the Levi-Civita connection $\nabla$ is compatible with the given Riemannian metric. Let $\Phi_G$ be a $G$-structure defined by a tensor field, and $\nabla_G$ be a compatible connection. We conjecture that $\Phi_G$ has a local presentation in terms of a Taylor polynomial approximation with Taylor coefficients that codify obstructions to integrability. Recall that $\Phi_G$ is integrable if it is locally equivalent to a canonical flat $G$-strucure on $\mathbb{R}^n$, typically denoted as $\Phi^{can}_G$ (e.g.\ $\Phi^{can}_{O(n)}$ is the $n\times n$ identity matrix). Recall as well that the torsion $Tor^{\nabla_G}$, curvature $R^{\nabla_G}$, and covariant derivatives $\nabla^k_G R^{\nabla_G} = 0$ are the obstructions to $\Phi_G$ being integrable. In the Riemannian case, $Tor^{\nabla} = 0$ and the metric is integrable iff $R^{\nabla} = 0.$

Can we use a framework similar to the one developed in this paper to introduce synthetic $G$-structures ? This would involve defining synthetic Lie groups. In view of the lack of topology, we expect for synthetic $G$-structures to exist quite abundantly. It would be interesting to find a generalization of Taylor's formula from Theorem \ref{RTHM_SM22} to such new structures. In order to do that, it would help to reinterpret Taylor's formula in Theorem \ref{RTHM_SM22} as a statement about synthetic $O(n)$-structures. 

\begin{conjecture}{(Taylor's theorem for synthetic $G$-structures)}\label{Gconjecture}
Suppose we had at our disposal good definitions of the synthetic torsion ($Tor^{\nabla_G}$) and synthetic curvature ($R^{\nabla_G}$) associated to a synthetic $G$-structure $\Phi_G$ on a set $M$. For any $p_0 \in M$, there exist a subset $U_{p_0} \subset M$ and a function $x\colon U_{p_0} \to \mathbb{R}$ such that the tensor components of the given structure satisfy, for any integer $k \geq 2$,

\begin{equation*}
    \begin{split}
        (\Phi_G)_{i_1 \dots i_l} \circ x^{-1} &= (\Phi^{can}_G)_{i_1 \dots i_l} + \sum_j a_j \big(Tor^{\nabla_G}(p_0)\big)x_j +\\
        & \sum_{2\leq |\alpha|\leq k} b_{\alpha} \big( (\nabla_G^{|\alpha|-2} R^{\nabla_G})(p_0)\big) x^{\alpha}+\\
        & O(\|x\|^{k+1}),
    \end{split}
\end{equation*}
where $a_j$, $b_{\alpha}$ are constants depending on the values of the torsion and covariant derivatives of the curvature at $p_0$. Note that we are using multi-index notation, so that $|\alpha|=\sum_i \alpha_i$ and $x^{\alpha} = \prod_i x^{\alpha_i}_i$, where $\alpha_i \geq 0$ is an integer.
\end{conjecture}

The goal would be to use this new synthetic theorem to obtain a classical differential geometry version, which is currently an open problem. Obviously some cases are known, notably the Riemannian one (Theorem \ref{RTHM}). There has also been some progress made for almost-complex structures (Proposition 7.1 \cite{Pali}), and almost-hermitian metrics (Lemma 8.2 \cite{Pali}, Theorem 4.8 \cite{JPD}).\\ 

\noindent
\textbf{4.\ On the proof of Theorem \ref{normal}.} Can one give a purely formal or synthetic proof of this theorem without going through the geometric construction of normal coordinates ?\\

\noindent
\textbf{3. Proof automation.} A challenging point of our formalization that was not addressed here is automation. It should be possible to achieve this via a combination of plugins, such as SMT solvers \cite{SMTS}, with LLM-powered versions of Rocq \cite{LLM}. For instance, how can we automate with sums in our proofs ?\\

\noindent
\textbf{5.\ Linguistic aspects of formalization.}
Is it possible to turn the slogan of section \ref{3.2} into a (meta)theorem ?

\section*{Conclusion}

We have demonstrated that it is possible to formalize local Riemannian geometry in Rocq through the example of Theorem \ref{RTHM}, which is about the second degree Taylor expansion of a metric in a normal coordinates. This was done by interpolating this well-known theorem and a not yet existent version of it written in the Calculus of Inductive Constructions, which would be Rocq formalization ready. The interpolation is a synthetic approximation that incorporates the existence of normal coordinates around any point on a Riemannian menifold, a result of global analysis, into an axiomatic system. 

As far as the theory goes, the synthetic approximation was achieved by introducing the notions of {\it $n$-dimensional synthetic local model of a Riemannian manifold} and {\it synthetic model of a theorem}. The latter furnishes a method for synthesizing a theorem by specification of the axioms and proof rules that are allowed. In our case, the synthetic model consists of Theorem \ref{RTHM_SM22}, and axioms that enforce the properties of (classical) normal coordinates that are needed to prove Theorem \ref{RTHM} using only a finite sequence of equational rewriting steps. In addition, we provide an alternative theoretical framework, the {\it category of rings of differential operators}, where less axioms are needed and where ring base change can be easily performed, opening the possibility to do, for example, local Riemannian geometry over a field of finite characteristic. 

In terms of Rocq implementation, we describe a technique for formalizing a synthetic model in stages that gradually unravel the content of axioms into propositions with proof. The formalization and synthetic model are inextricably connected; in fact, Theorem \ref{RTHM_SM22} and its proof were directly influenced by the formalization technique we developed. The crux of this technique is its ability to discern formal and analytical/topological components of a differential geometry result. 

Many interesting questions and new avenues of research have emerged from our work, related to both the synthetic theory and formalization process. Those are discussed in the final section of this article.

\vspace{1cm}

\end{document}